\documentclass[11pt,a4paper]{article}

\usepackage{amsmath,amssymb}                       
\usepackage{graphicx, caption, subfig}             
\usepackage{theorem}
\usepackage{cite}
\usepackage{a4wide}

\newcommand{\Real}{\mathbb{R}}
\newcommand{\Nat}{\mathbb{N}}

\newcommand{\Acal}{\mathcal{A}}
\newcommand{\Hcal}{\mathcal{H}}
\newcommand{\Ucal}{\mathcal{U}}

\newcommand{\rank}{\mathrm{rank}}
\newcommand{\ra}{\rightarrow}

\newcommand{\dd}[1]{\tfrac{\mathrm{d}}{\mathrm{d}#1}}

\newcommand{\bbm}{\begin{bmatrix}}
\newcommand{\ebm}{\end{bmatrix}}
\newcommand{\cas}[1]{\left\{\begin{split}#1\end{split}\right.}
\newcommand{\bit}{\begin{itemize}}
\newcommand{\eit}{\end{itemize}}

{\theorembodyfont{\itshape}\newtheorem{theorem}{Theorem}[section]}
{\theorembodyfont{\itshape}\newtheorem{lemma}[theorem]{Lemma}}
{\theorembodyfont{\itshape}\newtheorem{definition}[theorem]{Definition}}
{\theorembodyfont{\upshape}\newtheorem{remark}[theorem]{Remark}}
{\theorembodyfont{\upshape}}
{\theorembodyfont{\upshape}}
{\theorembodyfont{\upshape}\newtheorem{algorithm}[theorem]{Algorithm}}
{\theorembodyfont{\upshape}}

\newenvironment{proof}{{\emph{Proof:}}}{~\hfill$\square$}

\title{\LARGE \bf On the Minimum Attention and the Anytime Attention Control Problems for Linear Systems: A Linear Programming Approach\thanks{This work is partially supported by the Dutch Science Foundation (STW) and the Dutch Organization for Scientific Research (NWO) under the VICI grant ``Wireless controls systems: A new frontier in automation'', by the European 7th Framework Network of Excellence by the project ``Highly-complex and networked control systems (HYCON2-257462)'', and by the project ``Decentralised and Wireless Control of Large-Scale Systems (WIDE-224168)'', and by the National Science Foundation (NSF) award 0834771.}
\thanks{Tijs Donkers and Maurice Heemels are with the Hybrid and Networked Systems group of the Department of Mechanical Engineering of Eindhoven University of Technology, Eindhoven, the Netherlands, {\tt\footnotesize \{m.c.f.donkers, m.heemels\}@tue.nl}. Paulo Tabuada is with the Cyber-Physical Systems Laboratory of the Department of Electrical Engineering at the University of California, Los Angeles, CA, USA, {\tt\footnotesize tabuada@ee.ucla.edu}.}}

\author{M.C.F. Donkers \qquad P. Tabuada \qquad W.P.M.H. Heemels}

\begin{document}
\maketitle

\begin{abstract}
In this report, we present two control laws that are tailored for control applications in which computational and/or communication resources are scarce. Namely, we consider minimum attention control, where the `attention' that a control task requires is minimised given certain performance requirements, and anytime attention control, where the performance under the `attention' given by a scheduler is maximised. Here, we interpret `attention' as the inverse of the time elapsed between two consecutive executions of a control task. Instrumental for the solution will be a novel extension of the notion of a control Lyapunov function. By focussing on linear plants, by allowing for only a finite number of possible intervals between two subsequent executions of the control task and by taking the extended control Lyapunov function to be $\infty$-norm based, we can formulate the aforementioned control problems as linear programs, which can be solved efficiently online. Furthermore, we provide techniques to construct suitable $\infty$-norm-based (extended) control Lyapunov functions for our purposes. Finally, we illustrate the theory using two numerical examples. In particular, we show that minimum attention control outperforms an alternative implementation-aware control law available in the literature.
\end{abstract}

\section{Introduction}\label{sec5:introduction}

A current trend in control engineering is to no longer implement controllers on dedicated platforms having dedicated communication channels, but in embedded microprocessors and using (shared) communication networks. Since in such an environment the control task has to share computational and communication resources with other tasks, the availability of these resources is limited and might even be time-varying. Despite the fact that resources are scarce, controllers are typically still implemented in a time-triggered fashion, in which the control task is executed periodically. This design choice is motivated by the fact that it enables the use of the well-developed theory on sampled-data systems, e.g., \cite{che_fra_BOOK95,ast_wit_BOOK97}, to design controllers and analyse the resulting closed-loop systems. This design choice, however, leads to over-utilisation of the available resources and requires over-provisioned hardware, as it might not be necessary to execute the control task every period. For this reason, several alternative control strategies have been developed to reduce the required computation and communication resources needed to execute the control task.

Two of such approaches are event-triggered control, see, e.g., \cite{tab_TAC07,hee_san_bos_IJC08,hen_joh_cer_AUT08,lun_leh_AUT10}, and self-triggered control, see, e.g., \cite{vel_fue_mar_RTSS03,wan_lem_TAC09,maz_ant_tab_ECC09}. In event-triggered control and self-triggered control, the control law consists of two elements: namely, a feedback controller that computes the control input, and a triggering mechanism that determines when the control task should be executed. The difference between event-triggered control and self-triggered control is that in the former the triggering mechanism uses current measurements, while in the latter it uses predictions using previously sampled and transmitted data and knowledge on the plant dynamics, meaning that it is the controller itself that triggers the execution of the control task. Current design methods for event-triggered control and self-triggered control are emulation-based approaches, by which we mean that the feedback controller is designed for an ideal implementation, while subsequently the triggering mechanism is designed (based on the given controller). Since the feedback controller is designed before the triggering mechanism, it is difficult, if not impossible, to obtain an optimal design of the combined feedback controller and triggering mechanism in the sense that the minimum number of controller executions is achieved while guaranteeing stability and a certain level of closed-loop performance. Hence, no solution to the codesign problem currently exists.

An alternative way to handle limited computation and communication resources is by using so-called \emph{anytime control} methods, see, e.g., \cite{gup_que_CDC10,gre_fon_bic_TAC11,gup_CDC09}. These are control laws that are able to compute a control input, given a certain minimum amount of computation resources allotted by a scheduler, while providing a `better' control input whenever more computation resources are available. What is meant by `better', varies from computing more control inputs \cite{gup_CDC09}, computing more future control inputs \cite{gup_que_CDC10}, or computing the control input using a higher-order dynamical controller \cite{gre_fon_bic_TAC11}.

In this report, we consider two methodologies that are able to handle scarcity in computation and communication resources. The first methodology adopts  minimum attention control (MAC), see \cite{bro_CDC97}, in which the objective is to minimise the attention the control loop requires, i.e., MAC maximises the next execution instant, while guaranteeing a certain level of closed-loop performance. Note that this control strategy is similar to self-triggered control, where also the objective is to have as few control task executions as possible, given a certain closed-loop performance requirement. However, contrary to self-triggered control, MAC is typically not designed using emulation-based approaches in the sense that it does not require a separate feedback controller to be available before the triggering mechanism can be designed. Clearly, this joint design procedure is more likely to yield a (close to) optimal design than a sequential design procedure would. The second methodology proposed in this report is more in line with anytime control, as discussed above. Namely, by assuming that after each execution of the control task, the control input cannot be recomputed for a certain amount of time that is specified by a scheduler, anytime attention control (AAC) finds a control input that maximises the performance of the closed-loop system, given this time-varying computation constraint. This setting is realistic in many embedded and networked systems, where a real-time scheduler distributes the available resources among all tasks, and hence, determines online, the execution instants of the control task.

The control problems studied in this report are similar to the ones studied in \cite{ant_tab_CDC10}. However, by focussing on linear systems, we will propose an alternative approach to solve the control problems at hand. As was already observed in \cite{ant_tab_CDC10}, the MAC and the AAC problem are related and the same solution strategy can be used to solve both problems. We will also use the same solution strategy, yet a different one than used in \cite{ant_tab_CDC10}, to solve the both problems in this report. In the solution strategy we propose, we focus on linear plants, as already mentioned, and consider only a finite number of possible interexecution times. Furthermore, we will employ control Lyapunov functions (CLFs) that can be seen as an extension of the CLFs for sampled-data systems, which will enable us to guarantee a certain level of performance. These extended CLFs will first be formulated for general sampled-data systems and will later be particularised to $\infty$-norm-based functions, see, e.g., \cite{kie_ada_ste_TAC92,pol_TAC95}. Namely, by using $\infty$-norm-based extended CLFs, we can formulate both the MAC and the AAC problem as linear programs (LPs), which can be efficiently solved online, thereby alleviating the computational burden as experienced in \cite{ant_tab_CDC10}. Furthermore, we provide techniques to construct suitable $\infty$-norm-based (extended) control Lyapunov functions for the control objectives under consideration.
We will illustrate the theory using two numerical examples. In particular, we will show that MAC outperforms the self-triggered control strategy of \cite{maz_ant_tab_ECC09}.

The remainder of this report is organised as follows. After introducing the necessary notational conventions used in this report, we formulate the MAC and the AAC problem in Section \ref{sec5:model}. In Section \ref{sec5:mainidea}, we show how the control problems can be solved using extended CLFs, in Section \ref{sec5:lyapunov}, we show how to guarantee well-defined solutions, and, in Section \ref{sec5:LP}, we present computationally tractable algorithms to solve the control problems efficiently. Finally, the presented theory is illustrated using numerical examples in Section~\ref{sec5:example} and we draw conclusions in Section \ref{sec5:conclusion}. Appendix \ref{5sec:appendix} contains the proofs of the lemmas and theorems.

\subsection{Nomenclature}

The following notational conventions will be used. For a vector $x\in\Real^n$, we denote by $[x]_i$ its $i$-th element and by $\| x \|_p := \sqrt[p]{\sum_{i=1}^n |x_i|^p}  $ its $p$-norm, $p\in\Nat$, and by $\|x\|_\infty = \max_{i=\{1,\ldots,N\}} | x_i |$, its $\infty$-norm. For a matrix $A\in\Real^{n\times m}$, we denote by $[A]_{ij}$ its $i,j$-th element, by $A^\top\in\Real^{m\times n}$ its transposed and by $\|A\|_p := \max_{x\neq0} \tfrac{\|Ax\|_p}{\|x\|_p}$, its induced $p$-norm, $p\in\Nat\cup\{\infty\}$. In particular, $\|A\|_\infty := \max_{i\in\{1,\ldots,n\}} \sum_{j=1}^m | [A]_{ij} |$. We denote the set of nonnegative real numbers by $\Real_+ := [0,\infty)$, and for a function $f:\Real_+ \ra \Real^n$, we denote the limit from above for time $t\in\Real_+$ by $\lim_{s\downarrow t} f(s)$, provided that it exists. Finally, to denote a set-valued function $F$ from $\Real^n$ to $\Real^m$, we write $F:\Real^n\hookrightarrow\Real^m$, meaning that $F(x)\subseteq\Real^m$ for each $x\in\Real^n$.


\section{Problem Formulation} \label{sec5:model}

In this section, we formulate the minimum attention and the anytime attention control problem. To do so, let us consider a linear time-invariant (LTI) plant given~by
\begin{equation}
\dd{t}x = A x + B u, \label{eq5:plant}
\end{equation}
where $x\in\Real^{n_x}$ denotes the state of the plant and $u\in\Real^{n_u}$ the input applied to the plant. The plant is controlled in a sampled-data fashion, using a zero-order hold (ZOH), which leads to
\begin{equation}
u(t) = \hat{u}_k, \quad \text{for all } t\in[t_k,t_{k+1}),\label{eq5:controllaw}
\end{equation}
where the discrete-time control inputs $\hat{u}_k$, $k\in\Nat$, and the strictly increasing sequence of execution instants $\{t_k\}_{k\in\Nat}$ are given by either one of the solutions to the following two control problems:

\begin{itemize}
\item The minimum attention control (MAC) Problem: \emph{Find a set-valued function $F_{\mathrm{MAC}}:\Real^{n_x} \hookrightarrow \Real^{n_u}$ and a function $h:\Real^{n_x}\ra\Real_+$, such that
\begin{equation}
\cas{\hat{u}_k &\in F_{\mathrm{MAC}}(x(t_k)) \\ t_{k+1}&=t_k + h(x(t_k)),}\label{eq5:def_minimum}
\end{equation}
for all $k\in\Nat$, renders the plant \eqref{eq5:plant} with ZOH \eqref{eq5:controllaw} stable and guarantees a certain level of performance, both defined in an appropriate sense, while, for each $x\in\Real^{n_x}$, $h(x)$ is as large as possible.}

\item The anytime attention control (AAC) Problem: \emph{Find a set-valued function $F_{\mathrm{AAC}}:\Real^{n_x} \times \Real_+ \hookrightarrow \Real^{n_u}$, such that
\begin{equation}
\cas{\hat{u}_k &\in F_{\mathrm{AAC}}(x(t_k),h_k) \\ t_{k+1}&=t_k + h_k,}\label{eq5:def_anytime}
\end{equation}
for all $k\in\Nat$, renders the plant \eqref{eq5:plant} with ZOH \eqref{eq5:controllaw} stable and maximises performance in an appropriate sense, assuming that $h_k$, $k\in\Nat$, is given at time $t_k$ by the real-time scheduler.}
\end{itemize}

Note that the mappings $F_{\mathrm{MAC}}$ and $F_{\mathrm{AAC}}$ in the problems above are set-valued functions, i.e., $F_{\mathrm{MAC}}(x) \subseteq\Real^{n_u}$, for all $x\in\Real^{n_x}$, and $F_{\mathrm{AAC}}(x,h)\subseteq\Real^{n_u}$, for all $x\in\Real^{n_x}$ and $h\in\Real_+$. 
This means that $\hat{u}_k$, $k\in\Nat$, can be chosen from a subset $F_{\mathrm{MAC}}(x(t_k))$ or $F_{\mathrm{AAC}}(x(t_k),h_k)$ of $\Real^{n_u}$, while still guaranteeing the required properties of the MAC and the AAC problem.

To make the preceding problems well defined we need to give a precise meaning to the terms stability and performance qualifying the solutions of the closed-loop system given by \eqref{eq5:plant}, \eqref{eq5:controllaw}, with \eqref{eq5:def_minimum} or \eqref{eq5:def_anytime}.

\begin{definition}\label{th5:GES}
The system \eqref{eq5:plant}, \eqref{eq5:controllaw}, with \eqref{eq5:def_minimum} or \eqref{eq5:def_anytime}, is said to be \emph{globally exponentially stable} (GES) with a convergence rate $\alpha>0$ and a gain $c>0$, if for any initial condition $x(0)$, the corresponding solutions satisfy
\begin{equation}
\|x(t)\| \leqslant c e^{-\alpha t} \| x(0) \| , \label{eq5:stab_def}
\end{equation}
for all $t\in\Real_+$.
\end{definition}

The notion of performance used in this report is explicitly expressed in terms of the convergence rate $\alpha$ as well as the gain $c$. Only requiring a desired convergence rate $\alpha$ (in the MAC problem), or maximising it (in the AAC problem), could yield a very large gain $c$ and, thus, could yield unacceptable closed-loop behaviour. As we will show below (see Lemma \ref{th5:stab_a}), the guaranteed gain $c$ typically becomes large when the time between two controller executions, i.e., $t_{k+1}-t_k$, is large. Therefore, special measures have to be taken to prevent the gain $c$ from becoming unacceptably large.


\section[Formulating the Control Problems using CLFs]{Formulating the Control Problems using Control Lyapunov Functions}\label{sec5:mainidea}

In this section, we will propose a solution to the two considered control problems by formulating them as optimisation problems. In these optimisation problems, we will use an extension to the notion of a control Lyapunov function (CLF). Before doing so, we will briefly revisit some existing results on CLFs, see, e.g.,~\cite{son_SIAM83,kel_tee_SCL04}, and show how they can be used to design control laws that render the plant~\eqref{eq5:plant} with ZOH \eqref{eq5:controllaw} GES with a certain convergence rate $\alpha>0$ and a certain gain $c>0$.

\subsection{Preliminary Results on CLFs}

Let us now introduce the notion of a CLF, which has been applied to discrete-time systems in \cite{kel_tee_SCL04} and will now be applied to periodic sampled-data systems, given by the plant \eqref{eq5:plant} with ZOH \eqref{eq5:controllaw}, in which $t_{k+1}=t_k+h$, $k\in\Nat$, for some fixed $h>0$.

\begin{definition}\label{th5:stab_cond_a}
Consider the plant \eqref{eq5:plant} with ZOH \eqref{eq5:controllaw}. The function $V:\Real^{n_x}\ra\Real$ is said to be a \emph{control Lyapunov function} (CLF) for \eqref{eq5:plant} and \eqref{eq5:controllaw}, a convergence rate $\alpha>0$, a control gain bound $\beta>0$ and an interexecution time $h>0$, if there exist constants $\underline{a},\overline{a}\in\Real_+$ and $q\in\Nat$, such that for all $x\in\Real^{n_x}$
\begin{equation}
\underline{a} \| x \|^q \leqslant V(x) \leqslant \overline{a} \| x \|^q, \label{eq5:stab_cond}
\end{equation}
and, for all $x\in\Real^{n_x}$, there exists a control input $\hat{u} \in\Real^{n_u}$, satisfying $\|\hat{u}\|\leqslant \beta \|x\|$ and
\begin{equation}
V(e^{A h} x+\!\textstyle\int_0^{h}\!\!e^{As}\mathrm{d}s B \hat{u}) \leqslant e^{-\alpha q h} V(x).
\end{equation}
\end{definition}

Based on a CLF for a convergence rate $\alpha>0$, a control gain bound $\beta>0$ and an interexecution time $h>0$, as in Definition \ref{th5:stab_cond_a}, the control law
\begin{equation}
\cas{\hat{u}_k &\in F(x) := \{ u\!\in\Real^{n_u} |\, f(x,u,h,\alpha) \leqslant 0 \text{ and } \| u \| \leqslant \beta \|x\| \}, \\
     t_{k+1} &= t_k + h, }\label{eq5:controllaw1}
\end{equation}
in which
\begin{equation}
f(x,u,h,\alpha) := V(e^{Ah} x +\!\textstyle\int_0^{h}\!e^{As}B \mathrm{d}s \, u) - e^{-\alpha q h} V(x),\!\label{eq5:f}
\end{equation}
renders the plant \eqref{eq5:plant} with ZOH \eqref{eq5:controllaw} GES with a convergence rate $\alpha>0$ and a certain gain $c>0$, as we will show in the following lemma.

\begin{lemma}\label{th5:stab_a}
Assume there exists a CLF for \eqref{eq5:plant} with \eqref{eq5:controllaw}, a convergence rate $\alpha>0$, a control gain bound $\beta>0$ and an interexecution time $h>0$, in the sense of Definition \ref{th5:stab_cond_a}. Then, the control law \eqref{eq5:controllaw1} renders the plant \eqref{eq5:plant} with ZOH \eqref{eq5:controllaw} GES with the convergence rate $\alpha$ and the gain $c=\bar{c}(\alpha,\beta,h)$, where
\begin{equation}
\bar{c}(\alpha,\beta,h) := \sqrt[q]{\tfrac{\overline{a}}{\underline{a}}} \Big( e^{\|A\| h} + \beta \, \!\int_0^h \!\!e^{\|A\|s} \mathrm{d}s  \|B\| \Big) e^{\alpha h}.\label{eq5:constant_a}
\end{equation}
\end{lemma}

\begin{proof}
This lemma is a special case of Lemma \ref{th5:stab} that we will present and prove below.~
\end{proof}

Lemma~\ref{th5:stab_a} illustrates why it is important to express the notion of performance both in terms of the convergence rate $\alpha$ as well as the gain $c$, as was mentioned at the end of Section~\ref{sec5:model}. Namely, even though a CLF could guarantee GES with a certain convergence rate $\alpha$, for some control gain bound $\beta$ and for any arbitrarily large $h$, by using a corresponding CLF in the control law \eqref{eq5:controllaw1}, the consequence is that the guaranteed gain $c$ becomes extremely large, see Lemma \ref{th5:stab_a}. In particular, $c$ grows exponentially as $h$ becomes larger, which (potentially) yields undesirably large responses for large interexecution times $h = t_{k+1}-t_k$, $k\in\Nat$. To avoid having such unacceptable behaviour, we propose a control design methodology that is able to guarantee a desired convergence rate $\alpha$, as well as a desired gain $c$, even for large interexecution times $h$. This requires an extension of the CLF defined above.

\subsection{Extended Control Lyapunov Functions}

The observation that the interexecution time $h$ influences the gain $c$ is important to allow the MAC and the AAC problem to be formalised using CLFs. Namely, in order to achieve sufficiently high performance (meaning a sufficiently large $\alpha$ and a sufficiently small $c$), Lemma \ref{th5:stab_a} indicates that the interexecution time $h$ has to be selected sufficiently small. This, however, contradicts the MAC and the AAC problem, where in the former the interexecution time is to be maximised and in the latter it is time varying and specified by a scheduler. We therefore propose an extended control Lyapunov function (eCLF), which we will subsequently use to solve the MAC and the AAC problem. Roughly speaking, the eCLF is such that it does not only decrease from $t_k$ to $t_{k+1}$, but also from $t_k$ to intermediate time instants $t_k+\hbar_l$, for some (well-chosen) $\hbar_l>0$ satisfying $t_{k+1}-t_k>\hbar_l$, $k\in\Nat$, $l\in\{1,\ldots,L-1\}$. The existence of such an eCLF guarantees high performance, even though the interexecution time $\hbar_L := t_{k+1}-t_k$, $k\in\Nat$, can be large, as we will show after giving the formal definition of the eCLF.

\begin{definition}\label{th5:stab_cond}
Consider the plant \eqref{eq5:plant} with ZOH \eqref{eq5:controllaw}. The function $V:\Real^{n_x}\ra\Real$ is said to be an \emph{extended control Lyapunov function} (eCLF) for \eqref{eq5:plant} and \eqref{eq5:controllaw}, a convergence rate $\alpha>0$, a control gain bound $\beta>0$, and a set $\Hcal:=\{\hbar_1,\ldots,\hbar_L\}$, $L\in\Nat$, satisfying $\hbar_{l+1} > \hbar_{l} > 0$ for all $l\in\{1,\ldots,L-1\}$, if there exist constants $\underline{a},\overline{a}\in\Real_+$ and $q\in\Nat$, such that for all $x\in\Real^{n_x}$
\begin{equation}
\underline{a} \| x \|^q \leqslant V(x) \leqslant \overline{a} \| x \|^q \label{eq5:stab_cond_a}
\end{equation}
and, for all $x\in\Real^{n_x}$, there exists a control input $\hat{u} \in\Real^{n_u}$, satisfying $\|\hat{u}\|\leqslant \beta \|x\|$ and
\begin{equation}
V\big(e^{A \hbar_l} x\!+\!\textstyle\int_0^{\hbar_l}\!\!e^{As}\mathrm{d}s B \hat{u}\big) \leqslant e^{-\alpha q \hbar_l} V(x) \label{eq5:stab_cond_b}
\end{equation}
for all $l\in\{1,\ldots,L\}$.
\end{definition}

As before, based on an eCLF for a convergence rate $\alpha>0$, a control gain bound $\beta>0$ and a set $\Hcal$ as in Definition \ref{th5:stab_cond}, the control law
\begin{equation}
\!\!\cas{ \hat{u}_k &\in F(x) \!:=\! \big\{ u\!\in\Real^{n_u} |\, f(x,u,\hbar_l,\alpha) \leqslant 0 \, \forall\, l\in\{1,\ldots,L\} \text{ and } \|u\|<\beta \|x\| \big\}\!, \\
t_{k+1} &= t_k + \hbar_L, }\!\label{eq5:controllaw2}
\end{equation}
with $f(x,u,\hbar_l,\alpha)$ as defined in \eqref{eq5:f}, renders the plant \eqref{eq5:plant} with ZOH \eqref{eq5:controllaw} GES with a convergence rate $\alpha>0$ and a certain gain $c>0$ that is typically smaller than the gain obtained using an ordinary CLF, as we will show in the following lemma.

\begin{lemma}\label{th5:stab}
Assume there exists an eCLF for \eqref{eq5:plant} with \eqref{eq5:controllaw}, a convergence rate $\alpha>0$, a control gain bound $\beta>0$ and a set $\Hcal:=\{\hbar_1,\ldots,\hbar_L\}$, $L\in\Nat$, satisfying $\hbar_{l+1} > \hbar_{l} > 0$ for all $l\in\{1,\ldots,L-1\}$, in the sense of Definition \ref{th5:stab_cond}. Then, the control law \eqref{eq5:controllaw2} renders the plant \eqref{eq5:plant} with ZOH \eqref{eq5:controllaw} GES with the convergence rate $\alpha$ and the gain $c=\bar{c}(\alpha,\beta,\Delta_\hbar,\hbar_L)$, where
\begin{equation}
\bar{c}(\alpha,\beta,\Delta_\hbar,\hbar_L) := \sqrt[q]{\tfrac{\overline{a}}{\underline{a}}} \Big( e^{\|A\|\Delta_\hbar} + \beta e^{\alpha (\hbar_L - \Delta_\hbar)} \!\int_0^{\Delta_\hbar}\!\! e^{\|A\|s}\mathrm{d}s \|B\|\Big) e^{\alpha\Delta_\hbar},\label{eq5:constant}
\end{equation}
with $\Delta_\hbar:=\max_{l\in \{1,\hdots,L\}} (\hbar_l-\hbar_{l-1})$, in which  $\hbar_0:=~0$.
\end{lemma}

\begin{proof}
The proof can be found in Appendix \ref{5sec:appendix}.
\end{proof}

The existence of an eCLF for a well-chosen set $\Hcal$ (i.e., realising a sufficiently small $\Delta_\hbar$) guarantees high performance in terms of the convergence rate $\alpha$ and the gain $c$, while still allowing for large interexecution times $\hbar_L = t_{k+1}-t_k$, $k\in\Nat$. Indeed, by using the intermediate time instants $t_k+\hbar_l$, the gain $c$ in Lemma \ref{th5:stab} is generally much smaller than the gain $c$ in Lemma \ref{th5:stab_a}. However, making $\Delta_\hbar$ too small might lead to infeasibility of the control law, as decreasing $\Delta_\hbar$ for a fixed interexecution time $t_{k+1}-t_k$ means taking more intermediate times $\hbar_l$ and, thus, that more inequality constraints are added to the set-valued function $F$ in \eqref{eq5:controllaw2}, which, besides resulting in a much more complicated control law, might cause $F(x)=\emptyset$ for some $x\in\Real^{n_x}$. Hence, a tradeoff can be made between the magnitude of the gain $c$ and the number of constraints in $F(x)$ and we will exactly exploit this fact in the solution to the MAC and the AAC problem, as we will show below.

\subsection{Solving the MAC Problem using eCLFs}\label{th5:problem3}

We will now propose a solution to the MAC problem. As a starting point, we consider the control law \eqref{eq5:controllaw2}, which is based on an eCLF. Indeed, the existence of an eCLF for a convergence rate $\alpha>0$, a control gain bound $\beta>0$ and a set $\Hcal$ implies GES with convergence rate $\alpha$ and gain $c$ of the plant \eqref{eq5:plant} with ZOH \eqref{eq5:controllaw} and the control law \eqref{eq5:controllaw2}, according to Lemma~\ref{th5:stab}. However, given the function $V$, a convergence rate $\alpha$, a control gain bound $\beta$ and a set $\Hcal$, it might not always be possible to ensure that $F(x)\neq\emptyset$ for all $x\in\Real^{n_x}$. To resolve this issue, we take subsets of $\Hcal$ of the form $\Hcal_{\bar{L}} := \{\hbar_1,\ldots,\hbar_{\bar{L}}\}$, for $\bar{L}\in\{1,\ldots,L\}$, such that $\Hcal_1\subseteq\Hcal_2\subseteq\ldots\subseteq\Hcal_L=\Hcal$, and propose MAC, in which the objective is to maximise $\bar{L}\in\{1,\ldots,L\}$ for each given $x\in\Real^{n_x}$. In other words, for each given $x\in\Real^{n_x}$, $\bar{L}$ is maximised such that $F_{\bar{L}}(x)\neq\emptyset$, in which
\begin{align}
\!\!\!F_{\bar{L}}(x) := \big\{ u\!\in\Real^{n_u} |\,  f(x,u,\hbar_l,\alpha) \leqslant 0 \ \forall \ l\in\{1,\ldots,\bar{L}\} \text{ and } \|u\| \leqslant \beta \| x \| \big\},\!\label{eq5:minimum}
\end{align}
with $f(x,u,\hbar_l,\alpha)$ as defined in \eqref{eq5:f}. We maximise $\bar{L}$ to make the interexecution times $\hbar_{\bar{L}} = t_{k+1}-t_k$ maximal, yielding that the control law requires minimum attention. Hence, this MAC law is given by \eqref{eq5:def_minimum}, in which we take
\begin{equation}
\!\cas{F_{\mathrm{MAC}}(x) &:= F_{\bar{L}^\star(x)}(x)  \\
     h(x) &:= \hbar_{\bar{L}^\star(x)} }\!\label{eq5:minimum1}
\end{equation}
and
\begin{equation}
\bar{L}^\star(x) := \max \{ l\in\{1,\ldots,L\} \,|\, F_{l}(x) \neq \emptyset \}.\!\label{eq5:minimum2}
\end{equation}
Indeed, the control law \eqref{eq5:def_minimum} with \eqref{eq5:minimum1} and \eqref{eq5:minimum2} is a solution to the MAC problem, as every control input $\hat{u}_k$ is chosen such that the interexecution time $t_{k+1}-t_k = \hbar_{\bar{L}^\star(x(t_k))}$ is the largest one in the set $\Hcal$ for which $F_{\bar{L}^\star(x(t_k))}(x(t_k))\neq\emptyset$. Note that this control law is well defined if $F_{\mathrm{MAC}}(x)\neq\emptyset$, for all $x\in\Real^{n_x}$. This condition is equivalent to requiring that $F_1(x)\neq\emptyset$ for all $x\in\Real^{n_x}$. Namely, for each $x\in\Real^{n_x}$, it holds that $F_1(x)\supseteq F_2(x)\supseteq\ldots\supseteq F_L(x)$, which gives that, for each $x\in\Real^{n_x}$, $F_{\mathrm{MAC}}(x)\neq\emptyset$ implies that $F_1(x)\neq\emptyset$, while the fact that $F_1(x)\neq\emptyset$ implies that $F_{\mathrm{MAC}}(x)\neq\emptyset$ follows directly from \eqref{eq5:minimum1} and \eqref{eq5:minimum2}. Hence, \eqref{eq5:minimum1} is well defined if $F_1(x)\neq\emptyset$ for all $x\in\Real^{n_x}$, which is guaranteed if the function $V$ is an ordinary CLF for \eqref{eq5:plant} with \eqref{eq5:controllaw}, a convergence rate $\alpha>0$, a control gain bound $\beta>0$ and an interexecution time $\hbar_1$, in the sense of Definition \ref{th5:stab_a}.

We will now formally show that the proposed MAC law renders the plant \eqref{eq5:plant} with ZOH \eqref{eq5:controllaw} GES with convergence rate $\alpha$ and a certain gain $c$.

\begin{theorem}\label{th5:MAC}
Assume there exist a set $\Hcal:=\{\hbar_1,\ldots,\hbar_L\}$, $L\in\Nat$, satisfying $\hbar_{l+1} > \hbar_{l} > 0$ for all $l\in\{1,\ldots,L-1\}$, and an ordinary CLF for \eqref{eq5:plant} with \eqref{eq5:controllaw}, a convergence rate $\alpha>0$, a control gain bound $\beta>0$ and the interexecution time $\hbar_1$, in the sense of Definition \ref{th5:stab_cond_a}. Then, the MAC law \eqref{eq5:def_minimum}, with \eqref{eq5:f}, \eqref{eq5:minimum}, \eqref{eq5:minimum1} and \eqref{eq5:minimum2}, renders the plant \eqref{eq5:plant} with ZOH \eqref{eq5:controllaw} GES with the convergence rate $\alpha$ and the gain $c=\bar{c}(\alpha,\beta,\Delta_\hbar,\hbar_L)$ as in~\eqref{eq5:constant}.
\end{theorem}

\begin{proof}
The proof can be found in Appendix \ref{5sec:appendix}.
\end{proof}

\subsection{Solving the AAC Problem using eCLFs}\label{th5:problem4}

We will now propose a solution to the AAC problem, in which the objective is to `maximise performance' for an interexecution time $h_k$ given by the real-time scheduler at time $t_k$, $k\in\Nat$. The solution is again based on allowing only a finite number of possible interexecution times, i.e., $h_k\in\Hcal := \{\hbar_1,\ldots,\hbar_L\}$, $L\in\Nat$. Moreover, we consider only a finite number of possible convergence rates, i.e., $\alpha_k\in\Acal:=\{\bar\alpha_1,\ldots,\bar\alpha_J\}$, $k\in\Nat$, where each $\bar\alpha_{j+1} > \bar\alpha_j > 0$, $j\in\{1,\ldots,J-1\}$, $J\in\Nat$. A consequence of these choices is that the notion of `maximising performance' is actually relaxed to (approximately) maximising the local convergence rate $\alpha_k\in\Acal$ of the solutions of the closed-loop system \eqref{eq5:plant}, \eqref{eq5:controllaw} with \eqref{eq5:def_anytime}, in the sense that $V(x(t_{k+1})) \leqslant e^{-\alpha_k q h_k} V(x(t_k))$, for all $k\in\Nat$. In the proposed solution to the AAC problem, the local convergence rate $\alpha_k$, $k\in\Nat$ is maximised, by maximising $\bar{J}\in\{1,\ldots,J\}$, (so that $\bar\alpha_{\bar{J}}\in\Acal$ is maximised), while guaranteeing a certain gain $c$ (cf. Theorem \ref{th5:AAC}), for each given $x\in\Real^{n_x}$ and for each given $h\in\Hcal$. In other words, for each given $x\in\Real^{n_x}$ and each given $h\in\Hcal$, $\bar{J}$ is maximised such that $F_{\bar{L}(h),\bar{J}}(x)\neq\emptyset$, in which
\begin{align}
F_{\bar{L},\bar{J}}(x) := \{ u\in\Real^{n_u} \,|\, f(x,u,\hbar_l,\bar\alpha_{\bar{J}})\leqslant 0 \ \forall \ l\in\{1,\ldots,\bar{L}\} \text{ and } \|u\| \leqslant \beta \| x\| \}, \label{eq5:anytime}
\end{align}
with $f(x,u,\hbar_l,\alpha)$ as defined in \eqref{eq5:f} and where $\bar{L}(h)$ is a function that, for all $h\in\Hcal$, satisfies $\bar{L}(h)=\bar{L}$ if $h=\hbar_{\bar{L}}$. Hence, this AAC law is given by \eqref{eq5:def_anytime}, for a given value of $h\in\Hcal$ by the scheduler, where we take
\begin{equation}
F_{\mathrm{AAC}}(x,h)\!:= F_{\bar{L}(h),\bar{J}^\star(x,h)}(x),\!\label{eq5:anytime2}
\end{equation}
with
\begin{equation}
\bar{J}^\star(x,h) = \max \{ j\in\{1,\ldots,J\} \,|\, F_{\bar{L}(h),j}(x)\neq\emptyset \}. \label{eq5:anytime1}
\end{equation}
The control law \eqref{eq5:def_anytime}, with \eqref{eq5:anytime2} and \eqref{eq5:anytime1} is an AAC law, as for a given interexecution time, $t_{k+1}-t_k = h_k\in\Hcal$, a control control input $\hat{u}_k$ is chosen such that the local convergence rate $\alpha_k$ is maximal and a bound on the gain $c$ is guaranteed. Note that, similar to the solution to the MAC problem, this control law is well defined if $F_{\mathrm{AAC}}(x,h)\neq\emptyset$ for all $x\in\Real^{n_x}$ and all $h\in\Hcal$, which is equivalent to requiring that $F_{L,1}(x)\neq\emptyset$ for all $x\in\Real^{n_x}$. This is due to the fact that for each $x\in\Real^{n_x}$, for all $l_1,l_2\in\{1,\ldots,L\}$ and for all $j_1,j_2\in\{1,\ldots,J\}$, it holds that $F_{l_1,j_1}(x)\supseteq F_{l_2,j_2}(x)$, if $l_1\geqslant l_2$ and $j_1\leqslant j_2$, which means that, for each $x\in\Real^{n_x}$, $F_{\mathrm{AAC}}(x)\neq\emptyset$ implies that $F_{L,1}(x)\neq\emptyset$, while the fact that $F_{L,1}(x)\neq\emptyset$ implies that $F_{\mathrm{MAC}}(x)\neq\emptyset$ follows from that the fact that $F_{L,1}(x)\neq\emptyset$ implies that $F_{l,1}(x)\neq\emptyset$ for all $l\in\{1,\ldots,L\}$ and from \eqref{eq5:anytime2} and \eqref{eq5:anytime1}. Hence, \eqref{eq5:anytime2} is well defined if $F_{L,1}(x)\neq\emptyset$ for all $x\in\Real^{n_x}$, which is guaranteed if the function $V$ is an eCLF for \eqref{eq5:plant} with \eqref{eq5:controllaw}, a convergence rate $\alpha = \bar\alpha_1$, a control gain bound $\beta$ and the set~$\Hcal$.

We will now formally show that the proposed AAC law renders the plant \eqref{eq5:plant} with ZOH \eqref{eq5:controllaw} GES with at least convergence rate $\alpha = \bar\alpha_1$, and possibly a better convergence rate, and a certain gain~$c$.

\begin{theorem}\label{th5:AAC}
Assume there exist a set $\Acal:=\{\bar\alpha_1,\ldots,\bar\alpha_J\}$, $J\in\Nat$, satisfying $\bar\alpha_{j+1} > \bar\alpha_j > 0$ for all $j\in\{1,\ldots,J-1\}$,  and an eCLF for \eqref{eq5:plant} with \eqref{eq5:controllaw}, the convergence rate $\alpha = \bar\alpha_1$, a control gain bound $\beta$ and a set $\Hcal:=\{\hbar_1,\ldots,\hbar_L\}$, $L\in\Nat$, satisfying $\hbar_{l+1} > \hbar_{l} > 0$ for all $l\in\{1,\ldots,L-1\}$. Then, the AAC law \eqref{eq5:def_anytime}, with \eqref{eq5:f}, \eqref{eq5:anytime}, \eqref{eq5:anytime2} and \eqref{eq5:anytime1}, renders the plant \eqref{eq5:plant} with ZOH \eqref{eq5:controllaw} GES with (at least) the convergence rate $\alpha=\bar\alpha_1$ and the gain $c=\bar{c}(\bar\alpha_1,\beta,\Delta_\hbar,\hbar_L)$, as in \eqref{eq5:constant}.
\end{theorem}

\begin{proof}
The proof can be found in Appendix \ref{5sec:appendix}.
\end{proof}


\section{Obtaining Well-Defined Solutions} \label{sec5:lyapunov}

In this section, we will address the issue of how to guarantee that the solutions to the MAC and the AAC problem are \emph{well defined}, i.e., that $F_{\mathrm{MAC}}(x)\neq\emptyset$ for all $x\in\Real^{n_x}$ and that $F_{\mathrm{AAC}}(x,h)\neq\emptyset$ for all $x\in\Real^{n_x}$ and all $h\in\Hcal$. As was observed in the previous section, the existence of a CLF or an eCLF for \eqref{eq5:plant} with \eqref{eq5:controllaw}, a convergence rate $\alpha$, a control gain bound $\beta$ and, for the CLF, an interexecution time $h$, and, for the eCLF, a set $\Hcal$, ensures that the MAC law and the AAC law, respectively, are well defined. To obtain such a CLF or an eCLF, and to guarantee that the two control problems can be solved efficiently (as we will show in the next section), we focus in this section on $\infty$-norm-based (e)CLFs of the form
\begin{equation}
V(x) = \| P x \|_\infty, \label{eq5:lyapunov}
\end{equation}
with $P\in\Real^{m\times n_x}$ satisfying $\rank(P)=n_x$. Note that \eqref{eq5:lyapunov} is a suitable candidate (e)CLF, in the sense of Definition~\ref{th5:stab_cond}, with $q=1$, since \eqref{eq5:stab_cond} and \eqref{eq5:stab_cond_a} are satisfied with
\begin{equation}\label{eq5:lyapunov2}
\overline{a}=\|P\|_\infty, \quad \text{and} \quad \underline{a}=\max\{ a>0 \,|\, a \|x\| \leqslant \|Px\| \text{ for all } x\in\Real^{n_x} \}.
\end{equation}
In fact, $\rank(P)=n_x$ ensures that $\underline{a}>0$.

We will now provide a two-step procedure to obtain a suitable (e)CLF. The first step is to consider an auxiliary control law of the form
\begin{equation}
u(t) = K x(t) \label{eq5:controller}
\end{equation}
that renders the plant \eqref{eq5:plant} GES. To avoid any misunderstanding, \eqref{eq5:controller} is not the control law  being used; it is just an auxiliary control law that is useful to construct a candidate (e)CLF. The actual MAC law will be given by \eqref{eq5:def_minimum}, with \eqref{eq5:minimum1} and \eqref{eq5:minimum2}, and the AAC law will be given by \eqref{eq5:def_anytime}, \eqref{eq5:anytime2} and \eqref{eq5:anytime1} based on \eqref{eq5:lyapunov}, and neither one of these uses a matrix~$K$.

Using the auxiliary control law, we can find a Lyapunov function for the plant \eqref{eq5:plant} with control law \eqref{eq5:controller} (without ZOH \eqref{eq5:controllaw}) by employing the following intermediate result. This intermediate result can be seen as a slight extension of the results presented in \cite{kie_ada_ste_TAC92,pol_TAC95} to allow GES to be guaranteed, instead of only global asymptotic stability.

\begin{lemma}\label{th5:ct_lyap}
Assume that there exist a matrix $P\in\Real^{m\times n_x}$, with $\mathrm{rank}(P) = n_x$, a matrix $Q \in\Real^{m \times m}$ and a scalar $\hat\alpha>0$ satisfying
\begin{subequations}\label{eq5:ct_lyap}
\begin{align}
P (A+BK) - Q P &= 0 \\
 [Q]_{ii} + \!\!\!\sum_{j\in\{1,\ldots,m\}\backslash\{i\}}\!\!\! \big| [Q]_{ij} \big| &\leqslant -\hat\alpha,
\end{align}
\end{subequations}
for all $i\in\{1,\ldots,m\}$. Then, control law \eqref{eq5:controller} renders the plant \eqref{eq5:plant} GES with convergence rate $\hat\alpha$ and gain $\hat{c}=\overline{a}/\underline{a}$, with $\overline{a}$ and $\underline{a}$ as in \eqref{eq5:lyapunov2}.
\end{lemma}

\begin{proof}
The proof can be found in Appendix \ref{5sec:appendix}.
\end{proof}

Note that it is always possible, given stabilisability of the pair $(A,B)$, to find a matrix $P$ satisfying the hypotheses of Lemma \ref{th5:ct_lyap}, and constructive methods to obtain a matrix $P$ are given in \cite{kie_ada_ste_TAC92,pol_TAC95}. The second step in the procedure is to show that a matrix $P$ satisfying the conditions of Lemma \ref{th5:ct_lyap}, renders the plant \eqref{eq5:plant} with ZOH \eqref{eq5:controllaw} GES in case the auxiliary control law is given, for all $k\in\Nat$, by
\begin{equation}
\cas{ \hat{u}_k &= K x(t_k) \\ t_{k+1} &= t_k + h} \label{eq5:dt_controller}
\end{equation}
provided that $h~>~0$ is well chosen.

\begin{lemma}\label{th5:interevent}
Suppose the conditions of Lemma \ref{th5:ct_lyap} are satisfied. Then, for each $\alpha>0$ satisfying $\alpha<\hat\alpha$, the system given by \eqref{eq5:plant}, \eqref{eq5:controllaw} and \eqref{eq5:dt_controller} is GES with convergence rate $\alpha$ and gain $c=\bar{c}(\alpha,\|K\|,h)$ as in \eqref{eq5:constant_a}, for all $h<h_{\max}(\alpha)$ with
\begin{align}
\!h_{\max}(\alpha) = \min\!\big\{ \hat{h} >0 \, \big|  \|P (e^{A \hat{h}} + \!\textstyle\int_0^{\hat{h}}\!e^{As} \mathrm{d}s  BK ) (P^{\!\top}\!P)^{-1} P^{\!\top}\|_{\infty}\!> e^{-\alpha \hat{h}} \big\}.\!\label{eq5:dt_lyap}
\end{align}
\end{lemma}

\begin{proof}
The proof can be found in Appendix \ref{5sec:appendix}.
\end{proof}

Using the matrix $P$ and the function $h_{\max}(\alpha)$ obtained from Lemmas \ref{th5:ct_lyap} and \ref{th5:interevent}, we can now formally state the conditions under which the proposed solutions to the MAC and the AAC problem are well defined and how to achieve a desired convergence rate $\alpha$ and a desired gain~$c$.

\begin{theorem}\label{th5:cor1}
Assume there exist matrices $P\in\Real^{m\times n_x}$, $K\in\Real^{n_u\times n_x}$, and a scalar $\hat\alpha>0$ satisfying the conditions of Lemma \ref{th5:ct_lyap}, and let $0< \alpha <\hat\alpha$ and $c>\hat{c}$. If the control gain bound $\beta$ satisfies $\beta\geqslant\|K\|_\infty$ and the set $\Hcal:=\{\hbar_1,\ldots,\hbar_L\}$, $L\in\Nat$, is such that $\hbar_1<h_{\max}(\alpha)$ as in \eqref{eq5:dt_lyap}, and $c\geqslant\bar{c}(\alpha,\beta,\Delta_\hbar,\hbar_L)$ as in \eqref{eq5:constant}, then the MAC law \eqref{eq5:def_minimum}, with \eqref{eq5:f}, \eqref{eq5:minimum}, \eqref{eq5:minimum1}, \eqref{eq5:minimum2} and \eqref{eq5:lyapunov}, is well defined and renders the plant \eqref{eq5:plant} with ZOH \eqref{eq5:controllaw} GES with the convergence rate $\alpha$ and the gain $c$.
\end{theorem}

\begin{proof}
The proof can be found in Appendix \ref{5sec:appendix}.
\end{proof}

\begin{theorem}\label{th5:cor2}
Assume there exist matrices $P\in\Real^{m\times n_x}$, $K\in\Real^{n_u\times n_x}$, and a scalar $\hat\alpha>0$ satisfying the conditions of Lemma \ref{th5:ct_lyap}, and let $0< \alpha <\hat\alpha$ and $c>\hat{c}$ be given. If the control gain bound $\beta$, satisfies $\beta\geqslant\|K\|_\infty$, the set $\Acal:=\{\bar\alpha_1,\ldots,\bar\alpha_J\}$, $J\in\Nat$, is such that $\alpha\leqslant\bar\alpha_1<\hat\alpha$, the set $\Hcal:=\{\hbar_1,\ldots,\hbar_L\}$, $L\in\Nat$, is such that $\hbar_L<h_{\max}(\bar\alpha_1)$ as in \eqref{eq5:dt_lyap}, and $c\geqslant \bar{c}(\bar\alpha_1,\beta,\Delta_\hbar,\hbar_L)$ as in \eqref{eq5:constant}, then the AAC law \eqref{eq5:def_anytime}, with \eqref{eq5:f}, \eqref{eq5:anytime}, \eqref{eq5:anytime2}, \eqref{eq5:anytime1} and \eqref{eq5:lyapunov}, is well defined and renders the plant \eqref{eq5:plant} with ZOH \eqref{eq5:controllaw} GES with at least convergence rate $\alpha = \bar\alpha_1$, and possibly a better convergence rate, and a certain gain~$c$.
\end{theorem}

\begin{proof}
The proof can be found in Appendix \ref{5sec:appendix}.
\end{proof}

These theorems formally show how to choose the scalar $\beta$, and the sets $\Acal$ and $\Hcal$ to make each of the proposed solutions to the two control problems well defined and to achieve a desired convergence rate $\alpha$ and a desired gain~$c$.


\section[Making the Solutions Computationally Tractable]{Making the Solutions to the MAC and the AAC Problem Computationally Tractable} \label{sec5:LP}

As a final step in providing a complete solution to the MAC and the AAC problem, we will now propose computationally efficient algorithms to compute the control inputs generated by the MAC and AAC laws using online optimisation. To do so, note that the $\infty$-norm-based (e)CLFs as in \eqref{eq5:lyapunov} allow us to rewrite \eqref{eq5:f} as
\begin{equation}
f(x,u,h,\alpha) = \big\| P e^{Ah} x + \textstyle\int_0^{h} \!\!Pe^{As}B \mathrm{d}s \,u \big\|_\infty - e^{-\alpha h} \| Px \|_\infty.
\end{equation}
We can now observe that the constraint $f(x,u,h,\alpha)\leqslant0$, which appears in \eqref{eq5:minimum1} and \eqref{eq5:anytime2}, is equivalent to
\begin{equation}
\big| [P e^{A\hbar_l} x + \textstyle\int_0^{h} \!\!Pe^{As}B \mathrm{d}s \,u ]_i \big| - e^{-\alpha h} \| Px \|_\infty \!\leqslant 0,
\end{equation}
for all $i\in\{1,\ldots,m\}$, which is equivalent to $\bar{f}(x,u,h,\alpha)~\leqslant~0$, where
\begin{align}\label{eq5:constraints}
&\bar{f}(x,u,h,\alpha) := \bbm P e^{A h} x + P\!\!\textstyle\int_0^{h} e^{As} \mathrm{d}s B u \\ - P e^{A h} x - P\!\!\textstyle\int_0^{h} e^{As} \mathrm{d}s B u \ebm - e^{-\alpha h} \| Px \|_\infty {\tiny\bbm 1 \\ \vdots \\ \\ 1\ebm}
\end{align}
and the inequality is assumed to be taken elementwise, which results in $2m$ linear scalar constraints for $u$.

Equation \eqref{eq5:constraints} reveals that $\infty$-norm-based (e)CLFs convert the two considered problems into feasibility problems with linear constraints, allowing us to propose an algorithmic solution to the MAC and the AAC problem. The algorithms are based on solving the maximisation that appears in \eqref{eq5:minimum2} and \eqref{eq5:anytime1} by incrementally increasing $\bar{L}$ and $\bar{J}$, respectively.

\begin{algorithm}[Minimum Attention Control]\label{th5:algorithm1}
Let the matrix $P\in\Real^{m\times n_x}$, the scalars $\alpha,\beta>0$ and the set $\Hcal$, satisfying the conditions of Theorem \ref{th5:cor1}, be given. At each $t_k$, $k\in\Nat$, given state $x(t_k)$:
\begin{enumerate}
\item Set $l:=0$ and define $\Ucal^{\text{MAC}}_0 := \bigg\{ u\in\Real^{n_u} \, | \, {\scriptsize\bbm u \\ - u \ebm} - \beta \| x(t_k) \|_\infty {\tiny\bbm 1 \\ \vdots \\ 1 \ebm} \leqslant0 \bigg\}$
\item While $\Ucal^{\text{MAC}}_l \neq \emptyset$, and $l< L$
\bit
\item $\Ucal^{\text{MAC}}_{l+1} := \Ucal^{\text{MAC}}_l \cap \{ u\in\Real^{n_u} \, | \bar{f}(x(t_k),u,\hbar_{l+1},\alpha)\leqslant0 \}$
\item $l:=l+1$
\eit
\item If $l=L$ and $\Ucal^{\text{MAC}}_L \neq \emptyset$, take $\hat{u}_k \in \Ucal^{\text{MAC}}_{L}$, and $t_{k+1} = t_k + \hbar_{L}$
\item Or else, if $\Ucal^{\text{MAC}}_l = \emptyset$, take $\hat{u}_k \in \Ucal^{\text{MAC}}_{l-1}$, and $t_{k+1} = t_k + \hbar_{l-1}.$
\end{enumerate}
\end{algorithm}

\begin{algorithm}[Anytime Attention Control]\label{th5:algorithm2}
Let the matrix $P\in\Real^{m\times n_x}$, the scalar $\beta>0$, and the sets $\Acal$ and $\Hcal$, satisfying the conditions of Theorem \ref{th5:cor2}, be given. At each $t_k$, $k\in\Nat$, given state $x(t_k)$ and given $h_k \in \Hcal$, let $\bar{L}\in\{1,\ldots,L\}$ be such that $h_k=\hbar_{\bar{L}}$, and:
\begin{enumerate}
\item Set $j:=0$ and define $\Ucal^{\text{AAC}}_0 := \bigg\{ u\in\Real^{n_u} \, \bigg| \, {\scriptsize\bbm u \\ - u \ebm} - \beta \| x(t_k) \|_\infty {\tiny\bbm 1 \\ \vdots \\ 1 \ebm} \leqslant0 \bigg\}$
\item While $\Ucal^{\text{AAC}}_j \neq \emptyset$, and $j< J$,
\bit
\item $\Ucal^{\text{AAC}}_{j+1} := \Ucal^{\text{AAC}}_0 \cap \{ u\in\Real^{n_u} \, | \,  \bar{f}(x(t_k),u,\hbar_l,\alpha_{j+1})\leqslant 0 \ \forall \ l\in\{1,\ldots,\bar{L}\}\}$
\item $j:=j+1$
\eit
\item If $j=J$ and $\Ucal^{\text{AAC}}_J \neq \emptyset$, take $\hat{u}_k \in \Ucal^{\text{AAC}}_{J}$
\item Or else, if $\Ucal^{\text{AAC}}_j = \emptyset$, take $\hat{u}_k \in \Ucal^{\text{AAC}}_{j-1}.$
\end{enumerate}
\end{algorithm}

\begin{remark}
Since verifying that $\Ucal^{\text{MAC}}_l \neq\emptyset$, for some $l\in\{1,\ldots,L\}$, is a feasibility test for linear constraints, the algorithm can be efficiently implemented online using existing solvers for linear programs.
\end{remark}


\section{Illustrative Examples} \label{sec5:example}

In this section, we illustrate the presented theory using a well-known example in the NCS literature, see, e.g., \cite{wal_ye_CSM01}, consisting of a linearised model of a batch reactor. For this example, we solve both the MAC and the AAC problem. The linearised batch reactor is given by \eqref{eq5:plant} with
\begin{equation}
\left[ \begin{array}{c|c} \!\!A \!\!& \!\!B \!\!\!\end{array} \right] \!= \!\! { \left[ \begin{array}{@{}r@{\ \, }r@{\ \ }r@{\ \, }r@{\ }|@{} r@{\,}r@{}}
                                         1.380 &   -0.208 & 6.715  & -5.676 & 0     &  0 \\
                                        -0.581 &   -4.290 &  0     &  0.675 & \ 5.679 &  0 \\
                                         1.067 &    4.273 & -6.654 &  5.893 & 1.136 & \ -3.146 \\
                                         0.048 &    4.273 & 1.343  & -2.104 & 1.136 &  0 \end{array} \right]\!}.
\end{equation}

In order to solve the two control problems discussed in this report, we need a suitable (e)CLF. To obtain such a (e)CLF, we use the results from Section \ref{sec5:lyapunov}, and use an auxiliary control law \eqref{eq5:controller}, with
\begin{equation}
K = \left[ \begin{array}{@{}r@{\ \, }r@{\ \ }r@{\ \, }r@{}}
 0.0360 &  -0.5373 &  -0.3344 &  -0.0147 \\
 1.6301 &   0.5716 &   0.8285 &  -0.2821 \end{array} \right]
\end{equation}
yielding that the eigenvalues $A+BK$ are all real valued, distinct, and smaller than or equal to $-2$. This allows us to find a Lyapunov function of the form \eqref{eq5:lyapunov} using Lemma \ref{th5:ct_lyap}, with $P$ being the inverse of the matrix consisting of the eigenvectors of $A+BK$, $Q$ being a diagonal matrix consisting of the eigenvalues of $A+BK$, $\hat\alpha = 2$ and $\hat{c}\approx23.9$. This Lyapunov function will serve as an eCLF in the two control problems.

\begin{figure}[t]
\centering
\subfloat[Evolution of the states of the plant using MAC.]{\includegraphics[width = 0.45\columnwidth]{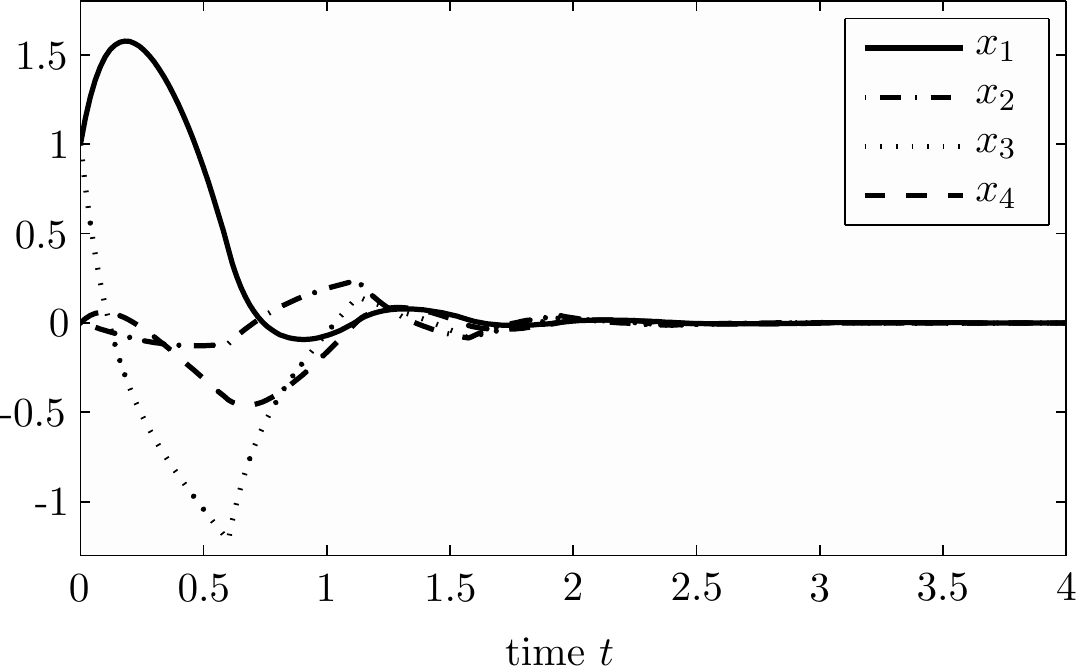}} \hfill
\subfloat[Evolution of the states of the plant using self-triggered control.]{\includegraphics[width = 0.45\columnwidth]{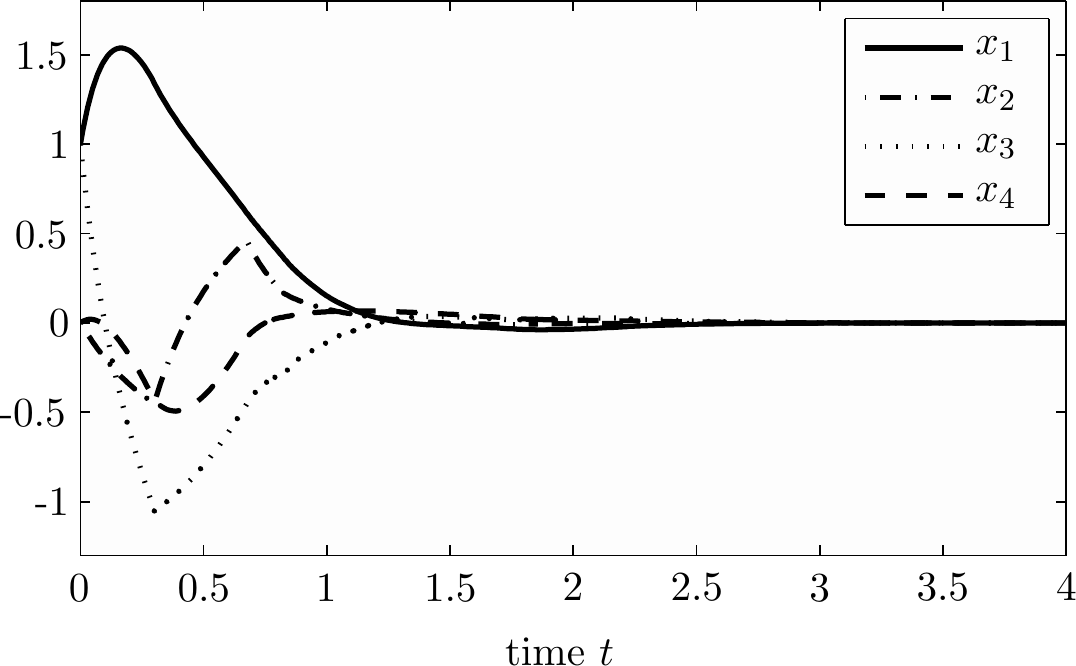}} \\
\subfloat[The decay of the Lyapunov function using MAC and self-triggered control.]{\includegraphics[width = 0.45\columnwidth]{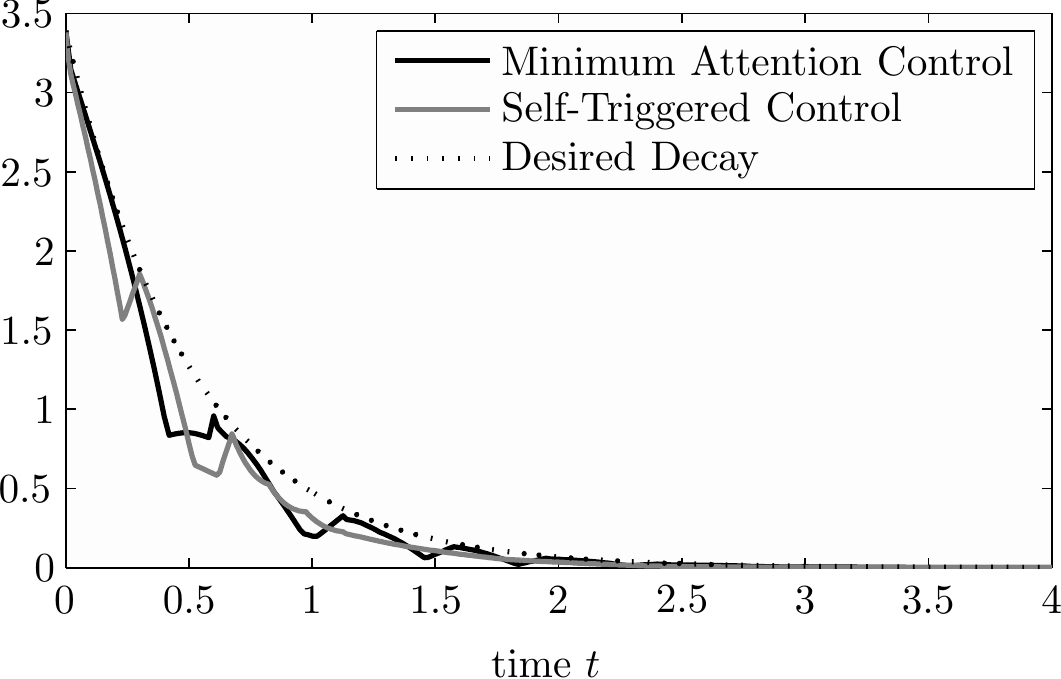}} \hfill
\subfloat[The interexecution times using MAC and self-triggered control.]{\includegraphics[width = 0.45\columnwidth]{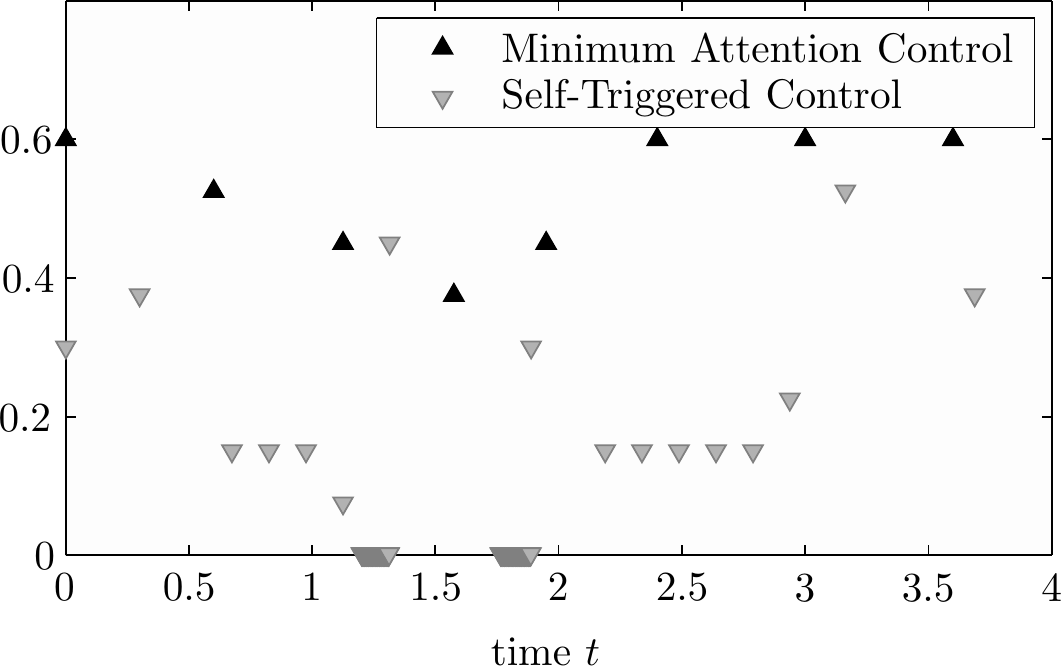}}
\caption{Minimum Attention Control.}\label{fig5:minimum}
\end{figure}

\subsection{The Minimum Attention Control Problem}

Given this eCLF, we can solve the MAC problem using Algorithm \ref{th5:algorithm1}. Before doing so, we use the result of Theorem \ref{th5:cor1} to guarantee that the MAC law is well defined and renders the closed-loop system GES with desired convergence rate $\alpha = 0.98 \hat\alpha = 1.96$ and desired gain $c=4\hat{c}\approx 95.7$. According to Theorem \ref{th5:cor1}, this convergence rate $\alpha$ and this gain $c$ can be achieved by taking $\beta = \| K \|_\infty \approx 3.1$ and
\begin{align}
\Hcal=\{\hbar_1,\ldots,\hbar_{10}\} =\{\tfrac{1.5}{1000},\tfrac{7.5}{100},\tfrac{15}{100},\tfrac{22.5}{100},\tfrac{30}{100},\tfrac{37.5}{100},\tfrac{45}{100},\tfrac{52.5}{100},\tfrac{60}{100}, \tfrac{67.5}{100}\},
\end{align}
because it holds that $\hbar_1<h_{\max}(\alpha)$ and that $\bar{c}(\alpha,\beta,\Delta_\hbar,\hbar_L)\leqslant c$. To implement Algorithm \ref{th5:algorithm1} in \textsc{Matlab}, we use the routine \verb"polytope" of the MPT-toolbox \cite{mpt}, to create the sets $\Ucal^{\text{MAC}}_l$, to remove redundant constraints and to check if the set $\Ucal^{\text{MAC}}_l$, $l\in\{1,\ldots,10\}$, is nonempty.

When we simulate the response of the plant with the resulting MAC law for the initial condition $x(0) = [1 \ 0 \ 1 \ 0 ]^{\!\top}$, we can observe that the closed-loop system is indeed GES, see Figure \ref{fig5:minimum}a, and satisfies the required convergence rate $\alpha$, see Figure \ref{fig5:minimum}c. To show the effectiveness of the theory, we compare our results with the self-triggered control strategy in the spirit of \cite{maz_ant_tab_ECC09}, however tailored to work with $\infty$-norm-based Lyapunov functions, resulting (by using the notation used in this report) in a control law \eqref{eq5:controllaw} with $\hat{u}_k=Kx(t_k)$, and $t_{k+1}=t_k + \hbar_{\bar{L}(x(t_k))}$, where
\begin{align}
&\bar{L}(x(t_k)) = \max\{ \hat{L}\in\{1,\ldots,L\} \,| \, f(x(t_k),Kx(t_k),\hbar_l,\alpha)\leqslant0 \ \forall \ l\in\{1,\ldots,\hat{L}\} \}.
\end{align}
To illustrate that also this control strategy renders the plant \eqref{eq5:plant} GES, we show the response of the plant to the initial condition $x(0) = [1 \ 0 \ 1 \ 0 ]^{\!\top}$ in Figure \ref{fig5:minimum}b, and the decay of the Lyapunov function in Figure \ref{fig5:minimum}c. Note that the decay of the Lyapunov function for MAC is comparable to the decay of the Lyapunov function for self-triggered control. However, when we compare the resulting interexecution times as depicted in Figure \ref{fig5:minimum}d, we can observe that the MAC yields much larger interexecution times. Hence, from a resource utilisation point of view, the proposed MAC outperforms the self-triggered control law.

\subsection{The Anytime Attention Control Problem}

Let us now illustrate the AAC problem, which can be solved using Algorithm~\ref{th5:algorithm2}. In this case, Theorem \ref{th5:cor2} provides conditions under which the AAC law is well defined and renders the closed-loop system GES with guaranteed convergence rate $\alpha = 0.5$ and gain $c=1.05\hat{c}\approx25$. According to Theorem \ref{th5:cor2}, this desired convergence rate $\alpha$ and this desired gain $c$ can be achieved by taking $\beta = 1.4 \| K \|_\infty \approx 4.6$, $\Acal = \{\bar\alpha_1,\ldots,\bar\alpha_{12}\}$, with $\bar\alpha_j = \tfrac{j}{2}$ for $j\in\{1,\ldots,12\}$, and $\Hcal = \{\hbar_1,\ldots,\hbar_{6}\} = \{ 0.011, 0.021, 0.031, 0.041, 0.051, 0.061\}$, because it holds that $\hbar_{6}<h_{\max}(\bar\alpha_1)$ and that $\bar{c}(\bar\alpha_1,\beta,\Delta_\hbar,\hbar_L)\leqslant c$.

When we simulate the response of the plant with the AAC law to the initial condition $x(0) = [1 \ 0 \ 1 \ 0 ]^{\!\top}$, and we take $h_k \in\Hcal$, where $h_k$, $k\in\Nat$, is given by an independent and identically distributed sequence of discrete random variables having a uniform probability distribution, we can observe that the closed-loop system is indeed GES, see Figure \ref{fig5:anytime}a. We also depict the corresponding realisation of $h_k$ for the interval $t\in[0,4]$ in Figure \ref{fig5:anytime}b. 
We conclude that AAC is able to yield high performance, even though the execution times are time-varying and given by a scheduler.

\begin{figure}[t]
\centering
\subfloat[Evolution of the states of the plant using AAC.]{\includegraphics[width = 0.45\columnwidth]{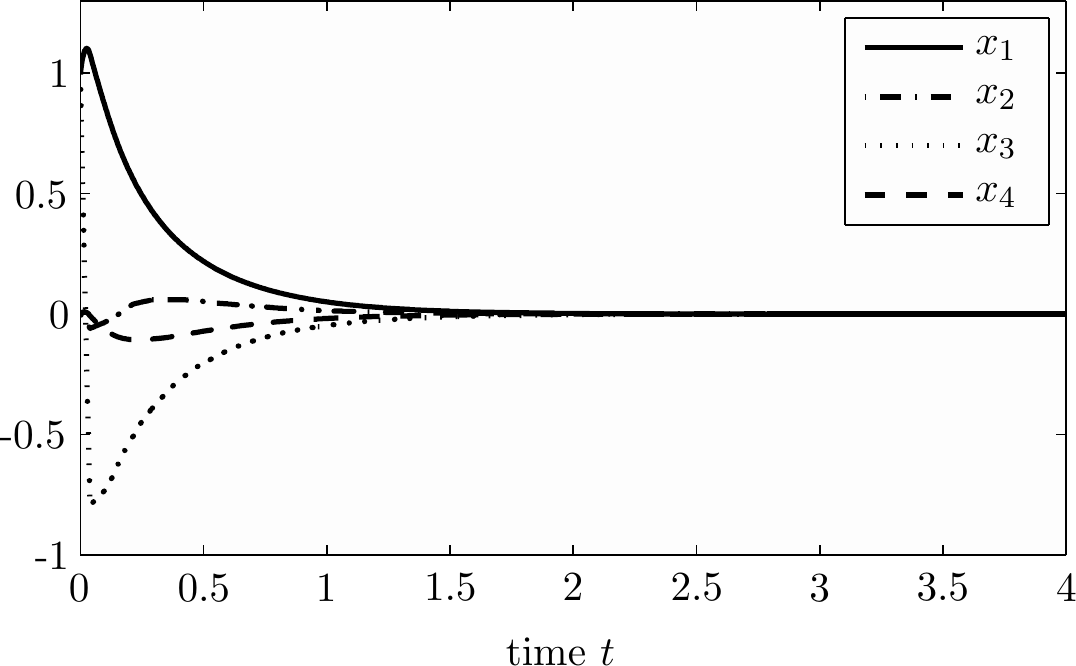}} \hfill
\subfloat[A realisation of the interexecution times for AAC.]{\includegraphics[width = 0.45\columnwidth]{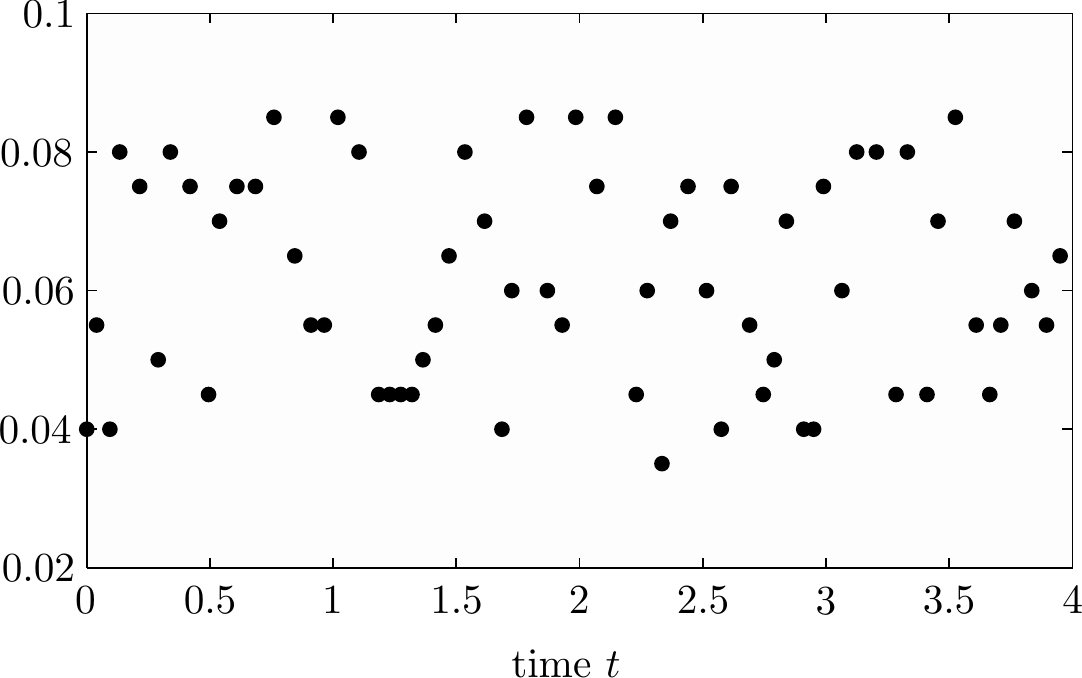}}
\caption{Anytime Attention Control.}\label{fig5:anytime}
\end{figure}


\section{Conclusion} \label{sec5:conclusion}

In this report, we proposed a novel way to solve the minimum attention and the anytime attention control problem. Instrumental for the solutions is a novel extension to the notion of a control Lyapunov function. We solved the two control problems by focussing on linear plants, by considering only a finite number of possible intervals between two subsequent executions of the control task and by choosing the extended control Lyapunov function (eCLF) to be $\infty$-norm-based, which allowed the two control problems to be formulated as linear programs. We provided a technique to obtain suitable eCLFs that render
the solution to the minimum attention control problem feasible with a guaranteed upper bound on the attention (i.e., an lower bound on the inter-execution times), while guaranteeing an \emph{a priori} selected performance level, and that renders solution to the anytime attention control problem feasible with a lower bound on the performance (in terms of a lower bound on the convergence rates), while guaranteeing a minimum level of performance. We illustrated the theory using two numerical examples. In particular, the first example showed that the proposed methodology outperforms a self-triggered control strategy that is available in the literature.

\appendix

\section{Proofs of Theorems and Lemmas}\label{5sec:appendix}

\paragraph{Proof of Lemma \ref{th5:stab}:}
Since \eqref{eq5:stab_cond_b} holds, and since the solutions to \eqref{eq5:plant} with \eqref{eq5:controllaw} satisfy
\begin{equation}
x(t_k+\hbar_l) = e^{A \hbar_l} x(t_k)\!+\!\textstyle\int_0^{\hbar_l}\!\!e^{As}B \mathrm{d}s \, \hat{u}_k,
\end{equation}
we have that
\begin{equation}
V(x(t_k+\hbar_l)) \leqslant e^{-\alpha q (t_k+\hbar_l)} V(x(0)). \label{eq5:stab_cond_1}
\end{equation}
for all $l\in\{0,\ldots,L-1\}$ and for all $t_k$, $k\in\Nat$, with $\hbar_0=0$. Now using \eqref{eq5:stab_cond_a}, we have that \eqref{eq5:stab_cond_1} implies
\begin{equation}
\| x(t_k+\hbar_l) \| \leqslant \sqrt[q]{\tfrac{\overline{a}}{\underline{a}}}  e^{-\alpha (t_k+\hbar_l)} \|x(0)\|, \label{eq5:stab_cond_2}
\end{equation}
for all $l\in\{0,\ldots,L-1\}$ and for all $t_k$, $k\in\Nat$, with $\hbar_0=0$. Moreover, because it holds that $\| \hat{u}_k \| \leqslant \beta \| x(t_k) \|$, the solutions to \eqref{eq5:plant} with \eqref{eq5:controllaw} also satisfy
\begin{align}
\| x(t) \| &\leqslant \| e^{A (t-t_k-\hbar_l)} \| \, \| x(t_k+\hbar_l) \| +\textstyle\int_{t_k+\hbar_l}^t \| e^{A(t-s)} \| \mathrm{d}s\, \| B \| \| \hat{u}_k \| \nonumber\\
           &\leqslant e^{\|A\| \Delta_\hbar} \, \|x(t_k\!+\!\hbar_l) \| + \beta \textstyle\int_0^{\Delta_\hbar} e^{\|A\|s} \mathrm{d}s\, \|B\|\, \| x(t_k) \|, \label{eq5:stab_cond_3}
\end{align}
for all $t\in[t_k+\hbar_l,t_k+\hbar_{l+1})$, $k\in\Nat$, $l\in\{0,\ldots,L-1\}$, with $\Delta_\hbar$ as defined in the hypothesis of the lemma. Substituting \eqref{eq5:stab_cond_2} into this expression (twice) yields
\begin{align}
\| x(t) \| & \leqslant \sqrt[q]{\tfrac{\overline{a}}{\underline{a}}} \Big( e^{\|A\| \Delta_\hbar} \, e^{-\alpha (t_k+\hbar_l)} + \beta\, \textstyle\int_0^{\Delta_\hbar} e^{\|A\|s} \mathrm{d}s \, \|B\| \, e^{-\alpha t_k}  \big) \|x(0)\|, \label{eq5:stab_cond_4}
\end{align}
for all $t\in[t_k+\hbar_l,t_k+\hbar_{l+1})$, $k\in\Nat$, $l\in\{0,\ldots,L-1\}$. Now realising that for all $t\in[t_k+\hbar_l,t_k+\hbar_{l+1})$, $k\in\Nat$, $l\in\{0,\ldots,L-1\}$, it holds that $e^{-\alpha (t_k+\hbar_l)} < e^{-\alpha t + \alpha\Delta_\hbar}$ and that $e^{-\alpha t_k} < e^{-\alpha t + \alpha \hbar_L }$ we have \eqref{eq5:stab_def} with $c$ as given in the hypothesis of the Lemma \ref{th5:stab}.\hfill~$\square$

~\\

\paragraph{Proof of Theorem \ref{th5:MAC}:}
Using the arguments given in Section \ref{th5:problem3}, we have that the hypotheses of the theorem guarantee that $F_{\mathrm{MAC}}(x)\neq\emptyset$ for all $x\in\Real^{n_x}$. By following a similar reasoning as done in the proof of Lemma \ref{th5:stab}, we can show that the MAC law guarantees that \eqref{eq5:stab_cond_4} holds for all $t\in[t_k+\hbar_l,t_k+\hbar_{l+1})$, $k\in\Nat$, $l\in\{0,\ldots,\bar{L}^\star(x(t_k))-1\}$, with $\hbar_0=0$, and all $x\in\Real^{n_x}$. Again realising that for all $t\in[t_k+\hbar_l,t_k+\hbar_{l+1})$, $k\in\Nat$, $l\in\{0,\ldots,\bar{L}^\star(x(t_k))-1\}$, it holds that $e^{-\alpha (t_k+\hbar_l)} < e^{-\alpha t + \alpha\Delta_\hbar}$ and that $e^{-\alpha t_k} < e^{-\alpha t + \alpha \hbar_{\bar{L}^\star(x(t_k))}} \leqslant e^{-\alpha t + \alpha \hbar_L}$ yields \eqref{eq5:stab_def} with gain $c=\bar{c}(\alpha,\beta,\Delta_\hbar,\hbar_L)$ as in~\eqref{eq5:constant}.\hfill~$\square$

~\\

\paragraph{Proof of Theorem \ref{th5:AAC}:}
Using the arguments given in Section \ref{th5:problem4}, we have that the hypotheses of the theorem guarantee that $F_{\mathrm{AAC}}(x,h)\neq\emptyset$ for all $x\in\Real^{n_x}$ and all $h\in\Hcal$. Moreover, as also argued in Section \ref{th5:problem4}, the proposed AAC law guarantees that the solutions of the system \eqref{eq5:plant}, \eqref{eq5:controllaw} with \eqref{eq5:def_anytime} satisfy
\begin{equation}
V(x(t_k+\hbar_l)) \leqslant e^{-\bar\alpha_{\bar{J}^\star(x(t_k),h_k)} q \hbar_l} V(x(t_k)) \leqslant e^{-\bar\alpha_{1} q \hbar_l} V(x(t_k))
\end{equation}
for all $l\in\{0,\ldots,\bar{L}(h_k)-1\}$, $k\in\Nat$. Now using the reasoning of the proof of Lemma \ref{th5:stab}, we can show that this implies that \eqref{eq5:stab_cond_4} holds for all $t\in[t_k+\hbar_l,t_k+\hbar_{l+1})$, $l\in\{0,\ldots,\bar{L}(h_k))-1\}$, with $\hbar_0=0$, $k\in\Nat$, and for all $x\in\Real^{n_x}$. Again realising that for all $t\in[t_k+\hbar_l,t_k+\hbar_{l+1})$, $l\in\{0,\ldots,\bar{L}(h_k)-1\}$, $k\in\Nat$, it holds that $e^{-\bar\alpha_1 (t_k+\hbar_l)}< e^{-\bar\alpha_1 t + \bar\alpha_1\Delta_\hbar}$ and that $e^{-\bar\alpha_1 t_k}< e^{-\bar\alpha_1 t + \bar\alpha_1 \hbar_{\bar{L}(h)}} \leqslant e^{-\bar\alpha_1 t + \bar\alpha_1 \hbar_L}$ yields \eqref{eq5:stab_def} with gain $c=\bar{c}(\bar\alpha_1,\beta,\Delta_\hbar,\hbar_L)$ as in \eqref{eq5:constant}.\hfill~$\square$

~\\

\paragraph{Proof of Lemma \ref{th5:ct_lyap}:}
The proof follows the same line of reasoning as in \cite{kie_ada_ste_TAC92,pol_TAC95}. GES of \eqref{eq5:plant} with \eqref{eq5:controller} with convergence rate $\hat\alpha$ and gain $\hat{c} = \overline{a}/\underline{a}$ is implied by the existence of a positive definite function, satisfying \eqref{eq5:stab_cond_a} and
\begin{equation}
\lim_{s\downarrow0} \tfrac{1}{s} \big( V(x(t+s)) - V(x(t)) \big) \leqslant -\hat\alpha V(x(t)), \label{eq5:interevent2}
\end{equation}
for all $t\in\Real_+$, which follows from the Comparison Lemma, see, e.g., \cite{khalil_1996}. Now using the fact that the solutions to \eqref{eq5:plant} with \eqref{eq5:controller} satisfy $\dd{t}x = (A+BK) x$, and using \eqref{eq5:lyapunov}, we obtain that \eqref{eq5:interevent2} is implied by
\begin{align}
\lim_{s\downarrow0} \tfrac{1}{s} ( \| P (I + s (A+BK)) x(t) \|_\infty - \| P x(t) \|_\infty ) \leqslant -\hat\alpha \| P x(t) \|_\infty,\!\label{eq5:interevent5}
\end{align}
for all $t\in\Real_+$. Using (\ref{eq5:ct_lyap}a), we have that, for all $t\in\Real_+$, \eqref{eq5:interevent5} implied by
\begin{align}
\!\!\lim_{s\downarrow0} \tfrac{1}{s} ( \| (I + s Q ) \|_\infty - 1 ) \| P x(t) \|_\infty \leqslant -\hat\alpha \| P x(t) \|_\infty,
\end{align}
which is, due to positivity of $\|Px\|_\infty$ for all $x\neq0$, equivalent to $\lim_{s\downarrow0} \tfrac{1}{s} ( \| (I + s Q ) \|_\infty - 1 ) \leqslant -\hat\alpha$, which is implied by (\ref{eq5:ct_lyap}b). This completes the proof.\hfill~$\square$

~\\

\paragraph{Proof of Lemma \ref{th5:interevent}:}
The proof is based on showing that the Lyapunov function obtained using Lemma \ref{th5:ct_lyap} also guarantees \eqref{eq5:plant} and \eqref{eq5:controllaw}, with \eqref{eq5:dt_controller} and $t_{k+1} = t_k + h$, $k\in\Nat$, to be GES with convergence rate $\alpha$ and gain $c:=\bar{c}(\alpha,\beta,h)$, where $\bar{c}(\alpha,\beta,h)$ as in \eqref{eq5:constant_a}, for all $h<h_{\max}(\alpha)$ as in \eqref{eq5:dt_lyap}. To do so, observe that the solutions of \eqref{eq5:plant} and \eqref{eq5:controllaw}, with \eqref{eq5:dt_controller} and $t_{k+1} = t_k + h$, $k\in\Nat$, satisfy
\begin{equation}
x(t) = ( e^{A (t-t_k)} + \textstyle\int_0^{t-t_k}e^{As} BK \mathrm{d}s ) x(t_k), \label{eq5:interevent1a}
\end{equation}
for all $t\in[t_k,t_k+h)$, $k\in\Nat$, which can be bounded as
\begin{align}
\| x(t) \| &\leqslant \big(e^{\|A\| h} + \textstyle\int_0^h e^{\|A\|s} \mathrm{d}s\, \|B\|\, \|K\| \big) \| x(t_k) \|,
\end{align}
for all $t\in[t_k,t_k+h)$, $k\in\Nat$. Now by following the ideas used in the proof of Lemma \ref{th5:stab}, and the candidate Lyapunov function of the form \eqref{eq5:lyapunov}, we have that GES with convergence rate $\alpha$ and gain $c$ of \eqref{eq5:plant} and \eqref{eq5:controllaw}, with \eqref{eq5:dt_controller} and $t_{k+1} = t_k + h$, $k\in\Nat$, is implied by requiring that
\begin{equation}
\| P x(t_k+h) \|_\infty - e^{-\alpha h} \| P x(t_k) \|_\infty \leqslant 0, \label{eq5:interevent1}
\end{equation}
for all $t_k$, $k\in\Nat$, and some well-chosen $h>0$. Substituting \eqref{eq5:interevent1a} and defining $\hat{x} := P x$, yielding $x = (P^\top P)^{-1}P^\top \hat{x}$, yields that that \eqref{eq5:interevent1} is implied by
\begin{align}
( \| P ( e^{A h}\!+\!\textstyle\int_0^h\!e^{As}BK \mathrm{d}s ) (P^{\!\top}P)^{-1}P^{\!\top}\|_\infty - e^{-\alpha h}) \| \hat{x}(t_k) \|_\infty \leqslant 0,
\end{align}
for all $\hat{x}(t_k)\in\Real^m$, which holds for all $h>0$, satisfying $h<h_{\max}(\alpha)$, as given in the hypothesis of the lemma, meaning that \eqref{eq5:interevent1} holds for all $\hat{x}(t_k)\in\Real^m$ and for all $h>0$, satisfying $h<h_{\max}(\alpha)$. This completes the proof.\hfill~$\square$

~\\

\paragraph{Proof of Theorem \ref{th5:cor1}:}
As a result of Lemma \ref{th5:interevent}, we have that the control input given by \eqref{eq5:dt_controller} renders the plant \eqref{eq5:plant} with ZOH \eqref{eq5:controllaw} GES with convergence rate $\alpha$ and gain $c:=\bar{c}(\alpha,\|K\|_\infty,h)$ as in \eqref{eq5:constant_a}, for any interexecution time $h<h_{\max}(\alpha)$ as in \eqref{eq5:dt_lyap}. To obtain a well-defined control law, we need that $F_{\mathrm{MAC}}(x)\neq\emptyset$, for all $x\in\Real^{n_x}$, which is guaranteed if and only if \eqref{eq5:minimum} satisfies $F_1(x)\neq\emptyset$ for all $x\in\Real^{n_x}$, as argued in Section \ref{th5:problem3}. This can be achieved by choosing $\beta\geqslant\|K\|_\infty$ and choosing the set $\Hcal:=\{\hbar_1,\ldots,\hbar_L\}$, $L\in\Nat$, such that $\hbar_1<h_{\max}(\alpha)$, as this yields that $F_1(x) \supseteq \{ K x \} \neq \emptyset$, if $V$ is chosen as in \eqref{eq5:lyapunov}. GES with the convergence rate $\alpha$ and the gain $c\geqslant\bar{c}(\alpha,\beta,\Delta_\hbar,\hbar_L)$ of \eqref{eq5:plant} with ZOH \eqref{eq5:controllaw} and \eqref{eq5:def_minimum}, with \eqref{eq5:f}, \eqref{eq5:minimum}, \eqref{eq5:minimum1}, \eqref{eq5:minimum2} and \eqref{eq5:lyapunov}, follows directly from Theorem \ref{th5:MAC}. This completes the proof.\hfill~$\square$

~\\

\paragraph{Proof of Theorem \ref{th5:cor2}:}
As a result of Lemma \ref{th5:interevent}, we have that the control input given by \eqref{eq5:dt_controller}, renders the plant \eqref{eq5:plant} with ZOH \eqref{eq5:controllaw} GES with a convergence rate $\alpha$, a gain $c:=\bar{c}(\alpha,\|K\|_\infty,h)$ as in \eqref{eq5:constant_a}, for any execution interval smaller than $h_{\max}(\alpha)$, as in \eqref{eq5:dt_lyap}. To obtain a well-defined control law, we need that $F_{\mathrm{AAC}}(x)\neq\emptyset$, for all $x\in\Real^{n_x}$, which is guaranteed if and only if \eqref{eq5:anytime} satisfies $F_{L,1}(x)\neq\emptyset$ for all $x\in\Real^{n_x}$, as argued in Section \ref{th5:problem4}. This can be achieved by choosing $\alpha\leqslant\bar\alpha_1<\hat\alpha$, the control gain bound $\beta\geqslant\|K\|_\infty$ and choosing the set $\Hcal:=\{\hbar_1,\ldots,\hbar_L\}$, $L\in\Nat$, such that $\hbar_L<h_{\max}(\alpha)$, as this yields that $F_{L,1}(x)\supseteq\{ K x \} \neq \emptyset$, if $V$ is chosen as in \eqref{eq5:lyapunov}. GES with the convergence rate $\alpha$ and the gain $c\geqslant \bar{c}(\bar\alpha_1,\beta,\Delta_\hbar,\hbar_L)$ of \eqref{eq5:plant} with ZOH \eqref{eq5:controllaw} and \eqref{eq5:def_anytime}, with \eqref{eq5:f}, \eqref{eq5:anytime}, \eqref{eq5:anytime2}, \eqref{eq5:anytime1} and \eqref{eq5:lyapunov}, follows directly from Theorem \ref{th5:AAC}. This completes the proof.\hfill~$\square$

\bibliographystyle{./IEEEtran}
\bibliography{CDCpaper}

\begin{thebibliography}{10}
\providecommand{\url}[1]{#1}
\csname url@samestyle\endcsname
\providecommand{\newblock}{\relax}
\providecommand{\bibinfo}[2]{#2}
\providecommand{\BIBentrySTDinterwordspacing}{\spaceskip=0pt\relax}
\providecommand{\BIBentryALTinterwordstretchfactor}{4}
\providecommand{\BIBentryALTinterwordspacing}{\spaceskip=\fontdimen2\font plus
\BIBentryALTinterwordstretchfactor\fontdimen3\font minus
  \fontdimen4\font\relax}
\providecommand{\BIBforeignlanguage}[2]{{%
\expandafter\ifx\csname l@#1\endcsname\relax
\typeout{** WARNING: IEEEtran.bst: No hyphenation pattern has been}%
\typeout{** loaded for the language `#1'. Using the pattern for}%
\typeout{** the default language instead.}%
\else
\language=\csname l@#1\endcsname
\fi
#2}}
\providecommand{\BIBdecl}{\relax}
\BIBdecl

\bibitem{che_fra_BOOK95}
T.~Chen and B.~A. Francis, \emph{Optimal Sampled-Data Control Systems}.\hskip
  1em plus 0.5em minus 0.4em\relax Springer-Verlag, 1995.

\bibitem{ast_wit_BOOK97}
K.~J. {\AA}str\"{o}m and B.~Wittenmark, \emph{Computer Controlled
  Systems}.\hskip 1em plus 0.5em minus 0.4em\relax Prentice Hall, 1997.

\bibitem{tab_TAC07}
P.~Tabuada, ``Event-triggered real-time scheduling of stabilizing control
  tasks,'' \emph{IEEE Trans. Autom. Control}, vol.~52, pp. 1680--1685, 2007.

\bibitem{hee_san_bos_IJC08}
W.~P. M.~H. Heemels, J.~H. Sandee, and P.~P.~J. van~den Bosch, ``Analysis of
  event-driven controllers for linear systems,'' \emph{Int. J. Control},
  vol.~81, pp. 571--590, 2008.

\bibitem{hen_joh_cer_AUT08}
T.~Henningsson, E.~Johannesson, and A.~Cervin, ``Sporadic event-based control
  of first-order linear stochastic systems,'' \emph{Automatica}, vol.~44, pp.
  2890--2895, 2008.

\bibitem{lun_leh_AUT10}
J.~Lunze and D.~Lehmann, ``A state-feedback approach to event-based control,''
  \emph{Automatica}, vol.~46, pp. 211--215, 2010.

\bibitem{vel_fue_mar_RTSS03}
M.~Velasco, J.~M. Fuertes, and P.~Marti, ``The self triggered task model for
  real-time control systems,'' in \emph{Proc. IEEE Real-Time Systems
  Symposium}, 2003, pp. 67--70.

\bibitem{wan_lem_TAC09}
X.~Wang and M.~Lemmon, ``Self-triggered feedback control systems with
  finite-gain {$\mathcal{L}_2$} stability,'' \emph{IEEE Trans. Autom. Control},
  vol.~45, pp. 452--467, 2009.

\bibitem{maz_ant_tab_ECC09}
M.~Mazo~Jr., A.~Anta, and P.~Tabuada, ``An {ISS} self-triggered implementation
  of linear controllers,'' \emph{Automatica}, vol.~46, pp. 1310--1314, 2010.

\bibitem{gup_que_CDC10}
V.~Gupta and D.~E. Quevedo, ``On anytime control of nonlinear processes though
  calculation of control sequences,'' in \emph{Proc. Conf. Decision \&
  Control}, 2010, pp. 7564--7569.

\bibitem{gre_fon_bic_TAC11}
L.~Greco, D.~Fontanelli, and A.~Bicchi, ``Design and stability analysis for
  anytime control via stochastic scheduling,'' \emph{IEEE Trans. Autom.
  Control}, 2011.

\bibitem{gup_CDC09}
V.~Gupta, ``On an anytime algorithm for control,'' in \emph{Proc. Conf.
  Decision \& Control}, 2009, pp. 6218--6223.

\bibitem{bro_CDC97}
R.~W. Brockett, ``Minimum attention control,'' in \emph{Proc. Conf. Decision \&
  Control}, 1997, pp. 2628--2632.

\bibitem{ant_tab_CDC10}
A.~Anta and P.~Tabuada, ``On the minimum attention and anytime attention
  problems for nonlinear systems,'' in \emph{Proc. Conf. Decision \& Control},
  2010, pp. 3234--3239.

\bibitem{kie_ada_ste_TAC92}
H.~Kiendl, J.~Adamy, and P.~Stelzner, ``Vector norms as {L}yapunov function for
  linear systems,'' \emph{IEEE Trans. Autom. Control}, vol.~37, no.~6, pp.
  839--842, 1992.

\bibitem{pol_TAC95}
A.~Pola\'{n}ski, ``On infinity norms as {L}yapunov functions for linear
  systems,'' \emph{IEEE Trans. Autom. Control}, vol.~40, no.~7, pp. 1270--1274,
  1995.

\bibitem{son_SIAM83}
E.~Sontag, ``A {L}yapunov-like characterization of asymptotic
  controllability,'' \emph{SIAM J. Control Optim.}, vol.~21, no.~3, pp.
  462--471, 1983.

\bibitem{kel_tee_SCL04}
C.~M. Kellett and A.~R. Teel, ``Discrete-time asymptotic controllability
  implies smooth control-lyapunov function,'' \emph{Syst. \& Control Lett.},
  vol.~51, pp. 349--359, 2004.

\bibitem{wal_ye_CSM01}
G.~Walsh and H.~Ye, ``Scheduling of networked control systems,'' \emph{IEEE
  Control Syst. Mag.}, vol.~21, no.~1, pp. 57--65, 2001.

\bibitem{mpt}
\BIBentryALTinterwordspacing
M.~Kvasnica, P.~Grieder, and M.~Baoti\'{c}, ``{Multi-Parametric Toolbox
  (MPT)},'' 2004. [Online]. Available: \url{http://control.ee.ethz.ch/~mpt/}
\BIBentrySTDinterwordspacing

\bibitem{khalil_1996}
H.~K. Khalil, \emph{Nonlinear Systems}.\hskip 1em plus 0.5em minus 0.4em\relax
  Prentice Hall, 1996.

\end{thebibliography}

\end{document}